\documentclass{amsart}

\newtheorem*{maintheorem}{Main Theorem}
\newtheorem{theorem}{Theorem}
\newtheorem{example}[theorem]{Example}
\newtheorem{claim}[theorem]{Claim}
\newtheorem{lemma}[theorem]{Lemma}
\newtheorem{problem}[theorem]{Problem}

\theoremstyle{definition}
\newtheorem{definition}[theorem]{Definition}

\newcommand{\w}{\omega}
\newcommand{\e}{\varepsilon}
\newcommand{\IR}{\mathbb R}
\newcommand{\IN}{\mathbb N}
\newcommand{\pr}{\mathrm{pr}}

\title[Haar-open sets and Steinhaus sum Theorem]{Haar-open sets: a right way of generalizing the Steinhaus sum theorem to non-locally compact groups}
\author{Taras Banakh}
\address{Ivan Franko National University of Lviv (Ukraine) and Jan Kochanowski University in Kielce (Poland)}
\email{t.o.banakh@gmail.com}

\begin{document}
\begin{abstract}
Let $X$ be the countable product of Abelian locally compact Polish groups and $A,B\subset X$ be two Borel sets, which are not Haar-null in $X$. We prove the then the sum-set $A+B:=\{a+b:a\in A,\;\;b\in B\}$ is {\em Haar-open} in  the sense that for any non-empty compact subset $K\subset X$ and point $p\in K$ there exists a point $x\in X$ such that the set $K\cap(A+B+s)$ is a neighborhood of $p$ in $K$.
This is a generalization of the classical Steinhaus Theorem (1920) to non-locally compact groups.
We do not know if this generalization holds for Banach spaces. 
\end{abstract}
\maketitle

In \cite{Stein} Steinhaus proved that for any Borel subsets $A,B$ of positive Haar measure in the real line the sum-set $A+B:=\{a+b:a\in A,\;b\in B\}$ has non-empty interior and the difference set $A-A:=\{a-b:a,b\in A\}$ is a neighborhood of zero in the real line. In \cite{Weil} Weil generalized this theorem of Steinhaus to all locally compact topological groups. The Steinhaus-Weil Theorem has many important applications, for example, to automatic continuity \cite{RS} or to functional  equations \cite{Kuczma2}.

In \cite{Ch} Christensen introduced the notion of Haar-null set in a topological group and generalized the difference part of the Steinhaus Theorem to non-locally compact Polish groups proving that for any Borel subset $A$ of a Polish Abelian group $X$ the difference $A-A$ is a neighborhood of zero in $X$ if the set $A$ is not Haar-null.

Following Christensen, we define a universally measurable subset $A$ of an Abelian topological group $X$ to be {\em Haar-null} if there exists a $\sigma$-additive Borel probability measure $\mu$ on $X$ such that $\mu(A+x)=0$ for all $x\in X$. A subset $A$ of a topological space $X$ is called {\em universally measurable} is it is measurable with respect to any Radon $\sigma$-additive Borel probability measure on $X$. A Borel probability measure $\mu$ on a topological space $X$ is called {\em Radon} if for any Borel subset $B\subset X$ and any $\e>0$ there exists a compact set $K\subset B$ such that $\mu(B\setminus K)<\e$. Any Radon measure $\mu$ on a topological space  is {\em regular}, which means that for any $\e>0$, any closed set $F\subset X$ has an open neighborhood $U\subset X$ such that $\mu(U\setminus F)<\e$. 

Let us observe that in contrast to the difference part of the Steinhaus Theorem, its sum-part does not generalize to non-locally compact groups.

\begin{example} The subset $A=[0,\infty)^\omega$ is not Haar-null in the Polish group $\IR^\w$, yet $A=A+A$ has empty interior in $\IR^\w$.
\end{example}

Nonetheless, in this paper we generalize the sum-part of the Steinhaus Theorem using a proper interpretation of non-empty interior for $A+B$ in non-locally compact case.

\begin{definition} A subset $A$ of a topological group $X$ is called {\em Haar-open} if for any compact subset $K\subset X$ and any point $x\in K$ there exists a point $y\in X$ such that the set $K\cap (A+y)$ is a neighborhood of $x$ in $K$.
\end{definition}
 
Haar-open sets in Abelian complete metric groups admit the following characterization.
 
 \begin{theorem} A subset $A$ of an Abelian complete metric group $X$ is Haar-open if and only if for any compact set $K\subset X$ there exists a point $x\in X$ such that the set $K\cap(A+x)$ has non-empty interior in $K$.
 \end{theorem}
 
 \begin{proof} The ``only if" part is trivial. To prove the ``if'' part, take any non-empty compact space $K\subset X$ and a point $p\in K$. Let $\rho$ be an invariant  complete metric generating the topology of the group $X$.
 
For every $n\in\w$ consider the compact set $K_n:=\{x\in K:\rho(x,p)\le\frac1{2^n}\}$ neighborhood of the point $p$ in $K$. Next, consider the compact space $\Pi:=\prod_{n\in\w}K_n$ and the map 
$$f:\Pi\to X,\;\;f:(x_n)_{n\in\w}\mapsto\sum_{n=0}^\infty x_n,$$
which is well-defined and continuous because of the completeness of the metric $\rho$ and the convergence of the series $\sum_{n=0}^\infty\frac1{2^n}$. 

If for the compact set $f(\Pi)$ in $X$ there exists $x\in X$ such that the intersection $f(\Pi)\cap(A+x)$ contains some non-empty open subset $U$ of $f(\Pi)$, then we can choose any point $(x_i)_{i\in\w}\in f^{-1}(U)$ and find $n\in\w$ such that $\{(x_i)\}_{i<n}\times \prod_{i\ge n}K_i\subset f^{-1}(U)$. Then for the point $s=\sum_{i<n}x_i\in \sum_{i<n}K_i\subset X$, we get $s+K_n\subset U\subset A+x$ and hence the intersection $K\cap (A+x-s)\supset K_n$ is a neighborhood of $p$ in $K$.
\end{proof}

It is clear that a subset of a locally compact topological group is Haar-open if and only if it has non-empty interior. That is why the following theorem can be considered as a generalization of the Steinhaus sum theorem.

\begin{maintheorem} Let $X:=\prod_{n\in\w}X_n$ be the countable product of locally compact  Abelian topological groups. If universally measurable sets $A,B\subset X$ are not Haar-null, then their sum $A+B$ is Haar-open in $X$.
\end{maintheorem}

This theorem follows from Lemmas~\ref{l7} and \ref{l8} proved below. In the proofs we shall use special measures, introduced in Definitions~\ref{d:stein} and \ref{d:inv}. For a Borel measure $\mu$ on a topological space $X$ a Borel set $B\subset X$ will be called {\em $\mu$-positive} if $\mu(B)>0$.

\begin{definition}\label{d:stein} A Borel probability measure $\mu$ on an Abelian topological group $X$  is called {\em Steinhaus-like} if 
\begin{enumerate}
\item $\mu(-B)=\mu(B)$ for any Borel set $B\subset X$ and
\item for any $\mu$-positive Borel sets $A,B\subset X$, the intersection $(A+a)\cap(B+b)$ is $\mu$-positive for some points $a,b\in X$.
\end{enumerate}
\end{definition}

\begin{definition}\label{d:inv} Let $K$ be a compact set in an Abelian  topological group $X$. A $\sigma$-additive Borel probability measure $\mu$ on $X$ is called {\em locally $K$-invariant} if there exists an increasing sequence $(M_n)_{n\in\w}$ of Borel sets in $X$ such that
\begin{enumerate}
\item $\lim_{n\to\infty}M_n=1$ and
\item for every $n\in\IN$ there exists a neighborhood $U\subset X$ of $\theta$ such that $\mu(B+x)=\mu(B)$ for any Borel set $B\subset M_n$ and any $x\in U\cap K$.
\end{enumerate}
\end{definition}

\begin{lemma}\label{l6}  If a Radon Borel probability measure $\mu$ on an Abelian topological group $X$ is locally $K$-invariant for some compact set $K\ni\theta$, then for any $\mu$-positive compact set $C$ the set $P:=\{x\in K:\mu\big(C\cap(C+x)\big)>0\}$ is a neigborhood of $\theta$ in $K$.
\end{lemma}

\begin{proof} Let $(M_n)_{n\in\w}$ be an increasing sequence of Borel sets, witnessing that the measure $\mu$ is locally $K$-invariant. 
Since $\lim_{n\to\infty}\mu(M_n)=1$, there exists $n\in\IN$ such that $\mu(C\cap M_n)>\frac34\mu(C)$ and hence $\mu(C')>\frac34\mu(C)$ for some compact subset $C'\subset C\cap M_n$. By Definition~\ref{d:inv}, for the number $n$ there exists a neighborhood $U\subset K$ of $\theta$ such that $\mu(B+x)=\mu(B)$ for any Borel set $B\subset M_n$ and any $x\in U$. By the regularity of the measure $\mu$, we can assume the neighborhood $U$ to be so small that $\mu(C'+U)<\frac54\mu(C')$. It remains to prove that $K\cap U\subset P$. Indeed, for any $x\in K\cap U$, the set $C'\subset M_n$ has measure $\mu(C')=\mu(C'+x)$. Since $C'+x\subset C'+U$, we conclude that
$$
\begin{aligned}
\mu\big(C\cap (C+x)\big)&\ge\mu\big(C'\cap (C'+x)\big)=\mu(C')+\mu(C'+x)-\mu\big(C'\cup(C'+x)\big)\ge 2\mu(C')-\mu(C'+U)>\\
&>2\tfrac34\mu(C)-\tfrac54\mu(C)=\tfrac14\mu(C)>0,
\end{aligned}
$$which means that $x\in P$ and $K\cap U$ is a neighborhood of $\theta$ in $K$.
\end{proof}

\begin{lemma}\label{l7} Let $X$ be an Abelian topological group such that for any compact set $K\subset X$ there exists a locally $K$-invariant Steinhaus-like probability Radon measure on $X$. Let $A,B$ be universally measurable sets in $X$. If $A,B$ are not Haar-null, then the set $A-B=\{a-b:a\in A,\;b\in B\}$ is Haar-open in $X$.
\end{lemma}

\begin{proof} To show that $A-B$ is Haar-open, fix any compact set $K$ in $X$ and a point $p\in K$. We need to find a point $s\in X$ such that $(A-B+s)\cap K$ is a neighborhood of $p$ in $K$. By our assumption, for the compact set $K-p$, there exists a locally $(K-p)$-invariant Steinhaus-like probability measure $\mu$ on $X$. Since $A,B$ are not Haar-null, there exist points $a,b\in X$ such that $\mu(A+b)>0$ and $\mu(B+b)>0$.  
Since the measure $\mu$ is Steinhaus-like, there exist points $a',b'\in X$ such that the set $(A+a+a')\cap (B+b+b')$ is $\mu$-positive and hence contains some $\mu$-positive compact set $C$. By Lemma~\ref{l6}, the set $P=\{x\in K-p:\mu\big(C\cap(C+x)\big)>0\}$ if a neighborhood of $\theta$ in $K-p$ and hence $U=P+p$ is a neighborhood of $p$ in $K$. It follows that for every $x\in U$, we get $C\cap (C+x-p)\ne\emptyset$ and hence $x\in C+p-C\subset (A+a+a')+p-(B+b+b')$. So, for the point $s:=a+a'+p-b-b'$ the set $A-B+s$ contains each point $x$ of the set $U\subset K$ and $K\cap(A-B+s)\supset U$ is a neighborhood of $p$ in $K$.
\end{proof}

\begin{lemma}\label{l8}  Let $X=\prod_{n\in\w}X_n$ be the Tychonoff product of locally compact Abelian topological groups. Then for every compact subset $K\subset X$ the group $X$admits a Steinhaus-like locally $K$-invariant $\sigma$-additive Borel measure with compact support.
\end{lemma}

\begin{proof} Fix a compact set $K\subset X$. Replacing $K$ by $K\cup\{\theta\}\cup(-K)$ we can assume that $\theta\in K=-K$.
 
For every $n\in\w$ identify the group $X_n$ with the subgroup $\{(x_i)_{i\in\w}\in X:x_n=\theta_n\}$ in $X$. Here by $\theta_n$ we denote the neutral element of the group $X_n$. Also we identify the products $\prod_{i< n}X_i$ and $\prod_{i\ge n}X_i$ with the subgroups $X_{<n}:=\{(x_i)_{i\in\w}\in X:\forall i<n\;\;x_i=\theta_i\}$ and $X_{\ge n}=\{(x_i)_{i\in\w}\in X:\forall i\ge n\;\;x_i=\theta_i\}$, respectively.

Let $\pr_n:X\to X_n$ and $\pr_{\le n}:X\to X_{\le n}$ be the natural coordinate projections. For every $n\in\w$ let $K_n$ be the projection of the compact set $K$ onto the locally compact group $X_n$.

Fix a decreasing sequence $(O_{i,n})_{i\in\w}$ of open neighborhoods of zero in the group $X_n$ such that  $\bar O_{0,n}$ is compact and $O_{i+1,n}+O_{i+1,n}\subset O_{i,n}=-O_{i,n}$ for every $i\in\w$. For every $m\in\w$ consider the neighborhood $U_{m,n}:=O_{1,n}+O_{2,n}+\dots+O_{m,n}$ of zero in the group $X_n$. By induction it can be shown that $-U_{mn,n}=U_{m,n}\subset O_{0,n}$ and $U_{m,n}\subset U_{m+1,n}$ for all $m\in\w$. Let $U_{\w,n}=\bigcup_{m\in\w}U_{m,n}$. 

Fix a Haar measure $\lambda_n$ in the locally compact group $X_n$.
Since the locally compact Abelian group $X_n$ is amenable, by the F\o lner Theorem \cite[4.13]{Pat}, there exists a compact set $\Lambda_n=-\Lambda_n\subset X$ such that $\lambda_n\big((\Lambda_n+K_n+U_{\w,n})\setminus \Lambda_n\big)<\frac1{2^n}\lambda_n(\Lambda_n)$. Multiplying the Haar measure $\lambda_n$ by a suitable positive constant, we can assume that $\lambda_n(\Omega_n)=1$ where $\Omega_n:=\Lambda_n+K_n+U_{\w,n}$. Now consider the probability  measure $\mu_n$ on $X_n$ defined by $\mu_n(B)=\lambda_n(B\cap \Omega_n)$ for any Borel subset of $X$. If follows from $\Omega_n=-\Omega_n$ that $\mu_n=-\mu_n$.

Let $\mu:=\otimes_{n\in\w}\mu_n$ be the product measure of the  
probability measures $\mu_n$. It is clear that $\mu=-\mu$. We claim that the measure $\mu$ is locally $K$-invariant and Steinhaus-like.

\begin{claim} The measure $\mu$ is locally $K$-invariant. 
\end{claim}

\begin{proof} For every $k\in\w$ consider the set 
$M_k:=\prod_{n<k}(\Lambda_n+K_n+U_{k,n})\times \prod_{n\ge k}\Lambda_n$. We claim that $\lim_{k\to\infty}M_k=1$.
Indeed, for every $\e>0$ we can find $m\in\IN$ such that $\prod_{n\ge m}(1-\frac1{2^n})>1-\e$ and then for any $n<m$ by the $\sigma$-additivity of the Haar measure $\lambda_n$, find $i_n> m$ such that $\lambda_n(\Lambda_n+K_n+U_{i_n,n})>(1-\e)^{1/m}$. Then for any $k\ge\max\limits_{n<m}i_n> m$, we obtain the lower bound 
$$
\begin{aligned}
\mu(M_k)&=\prod_{n<k}\lambda_n(\Lambda_n+K_n+U_{k,n})\cdot\prod_{n\ge k}\lambda_n(\Lambda_k)\ge \prod_{n<k}\lambda_n(\Lambda_n+K_n+U_{k,n})\cdot\prod_{n\ge m}\lambda_n(\Lambda_k)>\\
&>\prod_{n<m}(1-\e)^{\frac1m}\cdot\prod_{n\ge m}(1-\tfrac1{2^n})>(1-\e)^2.
\end{aligned}
$$

Now we show that the sequnce $(M_k)_{k\in\w}$ satisfies the second condition of Definition~\ref{d:inv}. Given any number $k\in\IN$, consider the the neighborhood $U:=\prod_{n<m}O_{k+1,n}\times\prod_{n\ge m}X_n$ of $\theta$ in $X$, and observe that for any $x\in U\cap K$, we have $M_k+x\subset\prod_{n\in\w}\Omega_n$. Consequently, for any Borel subset $B\subset M_k$ we get $\mu(B)=\mu(B+x)$ by the definition of the measure $\mu$.
\end{proof}

In the following claim we prove that the measure $\mu=-\mu$ is Steinhaus-like.

\begin{claim} For any $\mu$-positive Borel sets $A,B\subset X$ there are points $a,b\in X$ such that $\mu((A+a)\cap (B+b))>0$.
\end{claim}

\begin{proof} We lose no generality assuming that the sets $A,B$ are compact and are contained in $\prod_{n\in\w}\Omega_n$. By the regularity of the measure $\mu$, there exists a neighborhood $U\subset X$ of $\theta$ such that $\mu(A+U)<\frac65\mu(A)$ and $\mu(B+U)<\frac65\mu(B)$. Replacing $U$ by a smaller neighborhood, we can assume that $U$ is of the basic form $U=V+X_{<n}$ for some $n\in\w$ and some open set $V\subset X_{<n}$. Consider the tensor product $\lambda=\lambda'\otimes\mu'$ of the measures $\lambda'=\otimes_{k<n}\lambda_k$ and $\mu':=\otimes_{k\ge n}\mu_n$ on the subgroups $X_{<n}$ and $X_{\ge n}$ of $X$, respectively. It follows from $A\cup B\subset \prod_{n\in\w}\Omega_n$ that $\lambda(A)=\mu(A)$ and $\lambda(B)=\mu(B)$. Observe also the the measure $\lambda$ is $X_{<n}$-invariant in the sense that $\lambda(C)=\lambda(C+x)$ for any Borel set $C\subset X$ and any $x\in X_{<n}$.

Consider the projection $A'$ of $A$ onto $X_{<n}$ and for every $y\in X_{\ge n}$ let $A_y:=X_{<n}\cap(A-y)$ be a shifted $y$-th section of the set $A$.
Taking into account that $A_y\subset A'$, we conclude that $$\lambda'(A_y)\le\lambda'(A')=\mu(A'+X_{\ge n})\le\mu(A+U)<\tfrac65\mu(A).$$

By the Fubini Theorem, $\mu(A)=\int_{\mu'}\lambda'(A_y)dy>\frac56\mu(A'+X_{<n})=\frac56\lambda'(A')$. Consider the set  $$L_A:=\big\{y\in \prod_{k\ge n}\Omega_k:\lambda'(A_y)>\frac13\lambda'(A')\big\}$$ and observe that
$$\frac56\lambda'(A')<\mu(A)\le\mu'(L_A)\cdot \lambda'(A')+\frac13\lambda'(A')\cdot\big(1-\mu'(L_A)\big)$$ and hence $\mu'(L_A)>\frac34$.

By analogy, for the set $B$ consider the projection $B'$ of $B$ onto $X_{<n}$ and for every $x\in B'$ put $B_x:=X_{\ge n}\cap (B-x)$. Repeating the above argument we can show that the set $L_B:=\{y\in \prod_{k\ge n}\Omega_k:\mu'(B-x)>\frac13\lambda'(B')\}$ has measure $\mu'(L_B)>\frac34$.
Then $\mu'(L_A\cap L_B)>\frac12$.

Now we are ready to find $s\in X_{<n}$ such that $A\cap (B+s)\ne\emptyset$.
Observe that a point $z\in X$ belongs to $A\cap (B+s)$ is and only if $\chi_A(z)\cdot\chi_{B+s}(z)>0$, where $\chi_A$ and $\chi_{B+s}$ denote the characteristic functions of the sets $A$ and $B+s$ in $X$.

So, it suffices to show that the function $\chi_A\cdot \chi_{B+s}$ has a non-zero value. For this write $z$ as a pair $(x,y)\in X_{<n}\times X_{\ge n}$ and  for every $s\in X_{<n}$ consider the function 
$$g(s):=\int_{\lambda}\chi_A(z)\cdot\chi_{B+s}(z)dz=
\int_{\lambda'}\int_{\mu'}\chi_A(x,y)\cdot\chi_{B+s}(x,y)\,dy\,dx$$and its integral, transformed with help of Fubini's Theorem:
$$
\begin{aligned}
&\int_{\lambda'}g(s)\,ds=\int_{\lambda'}
\int_{\mu'}\int_{\lambda'}\chi_A(x,y)\cdot\chi_{B+s}(x,y)\,dx\, dy\,ds=\int_{\mu'}\int_{\lambda'}\int_{\lambda'}\chi_A(x,y)\cdot\chi_{B}(x-s,y)\,ds\, dx\, dy=\\
&\int_{\mu'}\int_{\lambda'}\chi_A(x,y)\int_{\lambda'}\chi_B(x-s,y)\,ds\,dx\, dy=\int_{\mu'}\int_{\lambda'}\chi_A(x,y)\lambda'(B_y)\,dx \,dy=\\
&\int_{\mu'}\lambda'(B_y)\int_{\lambda'}\chi_A(x,y)\,dx\, dy=\int_{\mu'}\lambda'(B_y)\cdot \lambda'(A_y)\,dy>\frac19\lambda'(A')\cdot\lambda'(B')\cdot \mu'(L_A\cap L_B)>0.
\end{aligned}
$$
Now we see that $g(s)>0$ for some $s\in X_{<n}$ and hence $A\cap (B+s)\ne\emptyset$.
\end{proof}
\end{proof}

Main Theorem raises many intriguing open problems.

\begin{problem} Is the conclusion of the Main Theorem true for Banach spaces? Is it true for the classical Banach space $c_0$ or $\ell_2$?
\end{problem}

A category version of the Haar-null set was recently introduced by Darji \cite{D}, who defined a Borel subset $B$ of a Polish Abelian group $X$ to be {\em Haar-meager} if there exists a continuous map $f:K\to X$ defined on a compact metrizable space such that $f^{-1}(A+x)$ is meager in $K$ for every $x\in X$. More information on Haar-meager sets can be found in \cite{BGJS}, \cite{EN}, \cite{J}. Observe that a closed subset of a Polish Abelian group is Haar-meager if and only if it is not Haar-open. 

\begin{problem} Let $A,B$ two Borel subsets of a Abelian Polish group $X\in\{\IR^\w,c_0,\ell_2\}$ such that $A,B$ are not Haar-meager in $X$. Is the sum-set $A+B$ (or  $A+A$) Haar-open?
\end{problem}

\end{document}